\documentclass[12pt]{amsart} 

\usepackage{amsmath}
\usepackage{amssymb}
\usepackage{amsthm}
\usepackage{bm}
\usepackage[all]{xy}

\newtheorem{thm}{Theorem}[section]

\newtheorem{lem}[thm]{Lemma}

\theoremstyle{definition}
\newtheorem{rem}{Remark}[section]

\makeatletter
\@namedef{subjclassname@2020}{\textup{2020} Mathematics Subject Classification}
\makeatother

\title[The associated groups of quandles]{Computing the associated groups of quandles with tools from group homology theory}
\author[K. Ishikawa]{Katsumi Ishikawa}
\address{Research Institute for Mathematical Sciences, Kyoto University, Kyoto 606-8502, Japan}
\email{katsumi@kurims.kyoto-u.ac.jp}
\date{}

\begin{document}
\begin{abstract}
We give a complete description of the associated group of any quandle as a central extension of the inner-automorphism group. As an application, we compute the second quandle homology groups of quandles of some families, including those of Alexander quandles.
\end{abstract}
\maketitle
%---------------------------------------------------------------------------------------------------------------------
\section{Introduction}
A quandle, introduced by Joyce \cite{joy} and Matveev \cite{mat}, is an algebraic system with three axioms, and (a conjugacy class of) a group with the conjugacy operation is a typical example. Quandles work well in knot theory, where various knot invariants have been discovered. For example, quandle $2$-cocycle invariants \cite{cjkls} and non-abelian cocycle invariants \cite{cegs} defined for cocycles on a quandle are rather simple and easy to calculate, but powerful and useful to distinguish knots and links. Here, different choices of cocycles give rise to different invariants even if the quandle $X$ is fixed, but an invariant obtained from the associated group ${\rm As}(X)$ is known to be universal among them (see \cite[Section 5.3]{nos1}). The associated group is the fundamental group of the quandle space \cite{nos0} of the quandle $X$, and it is known (\cite{eis2}) that we can compute the second quandle homology group $H_2^Q(X)$ from ${\rm As}(X)$. Thus, it is a quite important problem in quandle theory to determine the associated groups of quandles.
\par However, it is difficult to compute the associated group of a quandle in general. Although the computations for specific families of quandles are known (e.g., \cite{cla, bimnp} for some Alexander quandles), there has been no known procedure that works for any quandles. To solve this difficulty, some researchers consider the inner-automorphism group ${\rm Inn}(X)$. This group is more understandable than ${\rm As}(X)$, and there exists an epimorphism $\psi: {\rm As}(X) \to {\rm Inn}(X)$; since this is a central extension, finding the kernel is essential to investigate the structure of ${\rm As}(X)$. In fact, \cite{nos2} describes ${\rm Ker}\,\psi$ in terms of the group homology of ${\rm Inn}(X)$, though we need some conditions including the connectivity and the description is localized.
\par In this paper, we show a procedure to compute the structure of the associated group ${\rm As}(X)$ from group (co)homology without any condition or localization (Theorem \ref{main-thm}). The point is to consider an extension $\overline{\rm Inn}(X)$ of ${\rm Inn}(X)$ instead of ${\rm Inn}(X)$: Using the abelianization map $\alpha$ of ${\rm As}(X)$, we describe ${\rm Ker}(\alpha \times \psi)$ as a quotient of the second homology of the image $\overline{\rm Inn}(X)$ of $\alpha \times \psi$. Although we still have a problem of finding the second group homology, homology theory of groups is far more developed than that of quandles. Furthermore, it is often easier to calculate $H_2(\overline{\rm Inn}(X))$ than $H_2({\rm Inn}(X))$, which is another reason to focus on $\overline{\rm Inn}(X)$. 
%In fact, we can apply Theorem \ref{main-thm} to compute ${\rm As}(X)$ and $H_2^Q(X)$ for many quandles; see Section \ref{ap-sec}.
\par By Theorem \ref{main-thm}, we can compute and examine ${\rm As}(X)$ for many quandles. Consequently, we find the second quandle homology groups of (generalized) Alexander quandles. In what follows, for a $\mathbb{Z}[T^{\pm 1}]$-module $M$, or equivalently for an abelian group $M$ and an automorphism $T$ of $M$, we denote $M/(1-T)M$ by $M_T$ and ${\rm Ker}\,(1-T) \subset M$ by $M^T$.

\begin{thm}\label{alex-thm0}
Let $M$ be a $\mathbb{Z}[T^{\pm 1}]$-module and $X = Q_{M,T}$ the Alexander quandle. Then, we have
$$H_2^Q(X) \cong \left( (\Lambda^2 (1-T)M)_T \oplus ((1-T)M)^T \oplus \mathbb{Z}^{\oplus (|M_T| -1)} \right)^{\oplus M_T},$$
where $T$ acts on $\Lambda^2 (1-T)M$ diagonally.
\end{thm}

\begin{thm}\label{ga-thm0}
Let $G$ be a group and $\phi$ an automorphism of $G$. If the generalized Alexander quandle $Q_{G, \phi}$ is connected and the endomorphism $1- \phi$ on the abelianization $G_{\rm ab}$ of $G$ is invertible, then we have
$$H_2^Q(Q_{G,\phi}) \cong H_2(G)_{\phi}.$$
\end{thm}

\noindent %Here, for a module $M$ and an automorphism $T$, $M^T$ (resp. $M_T$) denotes the invariant (resp. coinvariant) module.
Moreover, we find a K$\ddot{\rm u}$nneth formula of second homology groups. In the following theorem, for a connected quandle $X$ we denote $\psi({\rm Ker}\, \alpha) \subset {\rm Inn}(X)$ by ${\rm Inn}_0(X)$.

\begin{thm}\label{prod-thm0}
Let $X,Y$ be connected quandles. Fixing any points $x_0 \in X$ and $y_0 \in Y$, we define an automorphism $\sigma$ of ${\rm Inn}_0(X) \times {\rm Inn}_0(Y)$ by $\sigma(f, g) = (s_{x_0}^{-1}fs_{x_0}, s_{y_0}^{-1} g s_{y_0})$, where $s_z$ denotes the right-multiplication map $\bullet * z$. Then, we have
$$H_2^Q(X\times Y) \cong H_2^Q(X) \oplus H_2^Q(Y) \oplus ({\rm Inn}_0(X)_{\rm ab} \otimes {\rm Inn}_0(Y)_{\rm ab})_\sigma.$$
\end{thm}

\par This paper is organized as follows. In Section \ref{pre-sec}, we review basic terms on quandles. In Section \ref{main-sec}, we state and prove Theorem \ref{main-thm}, which is the main theorem of the paper. After that, Section \ref{ap-sec} exposes applications of this theorem: We compute ${\rm Ker}(\alpha \times \psi)$ to obtain formulae for the second homology of quandles; proofs of the three theorems above are given in this section.
\par In this paper, we denote the homology group $H_*(G,\mathbb{Z})$ (resp. $H^Q_*(X,\mathbb{Z})$) of a group $G$ (resp. a quandle $X$) with trivial coefficient group $\mathbb{Z}$ by $H_*(G)$ (resp. $H^Q_*(X)$) for short. For example, $H_1(G)$ is isomorphic to the abelianization $G_{\rm ab}$ of $G$. Also, given a right action of a group $G$ on a set $S$, we denote the isotropy subgroup $\{ g \in G \mid s \cdot g = s \}$ of $s \in S$ by ${\rm Stab}_G(s).$
%---------------------------------------------------------------------------------------------------------------------
\section{Preliminaries}\label{pre-sec}
In this section, we review basic definitions on quandles. For further information, refer to \cite{el-ne,nos1}, for example. We also recall Pontryagin products of groups. In this paper we do not recall definitions and basic properties on group (co)homology in detail; see, e.g., \cite{bro} if needed.
\par Let $X$ be a set and $*$ a binary operation on $X$. Let us consider the following three conditions on the pair $(X,*)$:
\begin{itemize}
\item[(Q1)] We have $x * x = x$ for any $x \in X$.
\item[(Q2)] The map $s_x: X \to X$ defined by $s_x(a) = a * x$ is bijective for any $x \in X$.
\item[(Q3)] We have $(x * y) * z = (x*z) * (y*z)$ for any $x,y,z \in X.$
\end{itemize}
A pair $(X,*)$ satisfying (Q2) and (Q3) above is called a \textit{rack}. If a rack satisfies (Q1), it is called a \textit{quandle}.
\par Let $X$ be a rack. An \textit{automorphism} of $X$ is a bijection $f: X \to X$ satisfying $f(x * y) = f(x) * f(y)$ for any $x, y \in X$. By the conditions (Q2) and (Q3) in the definition above, the right-multiplication map $s_x$ is an automorphism. Then, we call the subgroup of the automorphism group generated by $s_x$ with all $x \in X$ the \textit{inner-automorphism group} and denote it by ${\rm Inn}(X)$. In this paper, the group multiplication in ${\rm Inn}(X)$ is given by the opposite composition; we have a right action of ${\rm Inn}(X)$ on $X$. We denote the set of the orbits by $\mathcal{O}_X$ and call each orbit a \textit{connected component}. If the action is transitive, $X$ is said to be (\textit{algebraically}) \textit{connected}.
\par The \textit{associated group} ${\rm As}(X)$ of a rack $X$ is defined by the group presentation
$$\langle e_x \text{ with }x \in X \mid e_x e_y = e_y e_{x*y}\text{ for } x, y \in X \rangle.$$
We can easily check that $a \cdot e_x = a * x$ defines a right action of ${\rm As}(X)$ on $X$, i.e., the correspondence $e_x \mapsto s_x$ defines a group homomorphism $\psi: {\rm As}(X) \to {\rm Inn}(X)$, which is surjective by the definition.
\par For a quandle $X$, the \textit{quandle homology groups} $H_*^Q(X)$ are defined. Although we compute $H_2^Q(X)$ for some quandles in Section \ref{ap-sec}, we do not need an explicit formulation and hence omit the definition; see, e.g., \cite{cjkls,el-ne,nos1} for details. Instead, we recall a result of Eisermann \cite{eis2}:

\begin{thm}[\cite{eis2}]\label{eis-thm}
Let $X$ be a quandle and $E \subset X$ a representative set of $\mathcal{O}_X$. Then, we have
$$H_2^Q(X) \cong \bigoplus_{x_0 \in E} ({\rm Stab}_{{\rm As}(X)} (x_0) \cap {\rm Ker}\,\epsilon)_{\rm ab},$$
where $\epsilon: {\rm As}(X) \to \mathbb{Z}$ is the group epimorphism that sends each generator $e_x \in {\rm As}(X)$ to $1 \in \mathbb{Z}$.
\end{thm}
%\par Furthermore, let us review homology theories on racks/quandles briefly; see, e.g., \cite{cjkls,el-ne} for details. Let $X$ be a rack. Let $C_n^R(X)$ ($n \geq 0$) be the free abelian group generated by $X^n$. For $(x_1, \dots, x_n) \in C_n^R(X)$, we set
%\begin{multline*}
%\partial(x_1, \dots, x_n) = \sum_{i=1}^n (-1)^n\bigl((x_1, \dots, x_{i-1}, x_{i+1}, \dots, x_n)\\
%-(x_1 * x_i, \dots, x_{i-1} * x_i, x_{i+1}, \dots, x_n)\bigr).
%\end{multline*}
%Then we obtain a chain complex $(C_*^R(X), \partial)$ and denote the homology groups by $H_*^R(X)$. Also, if $X$ is a quandle, we can check that subgroups
%$$C_0^D(X) = 0, \quad C_n^D(X) = \text{span}_{\mathbb{Z}}\{(x_1, \dots, x_n) \in C_n^R(X) \mid \text{$x_i = x_{i+1}$ for some $i$}\}$$
%of $C_*^R(X)$ form a subcomplex. Let $C_n^D(X)$ denote the quotient $C_n^R/C_n^D$. The \textit{quandle homology groups} $H_*^Q(X)$ of $X$ is the homology groups of the quotient complex $(C_*^Q(X),\partial)$.
\par Next, we shall recall the definition of Pontryagin products $\langle -,- \rangle$, following \cite{ed-li}. We fix a generator $[\mathbb{Z} \times \mathbb{Z}]$ of $H_2(\mathbb{Z} \times \mathbb{Z}) \cong \mathbb{Z}$. Let $G$ be a group. For a commutative pair of elements $g, h \in G$, we define the \textit{Pontryagin product} to be $\langle g, h \rangle = f_*[\mathbb{Z} \times \mathbb{Z}] \in H_2(G),$ where $f:\mathbb{Z} \times \mathbb{Z} \to G$ is the group homomorphism such that $f(1,0) = g$ and $f(0,1) = h$. The homology class $\langle g, h \rangle$ is represented by the inhomogeneous group cycle $[g | h] - [h | g]$. If $h = h_1 \cdots h_k$ for some $h_1, \dots, h_k \in G$, the cycle
$$\sum_{i=0}^{k-1} ([g_i | h_{i+1}] - [h_{i+1}| g_{i+1}])$$
also represents $\langle g, h \rangle$, where $g_i = (h_1 \cdots h_i)^{-1} g (h_1 \cdots h_i)$.
\par As basic properties of the Pontryagin products, we can easily show
\begin{itemize}
\item[(P1)] $\langle g,g \rangle = \langle g, 1_G \rangle = 0,$
\item[(P2)] $\langle g,hh' \rangle = \langle g,h \rangle + \langle g,h' \rangle$ if $[g,h] = [g,h'] = 1_G,$ and
\item[(P3)] $\langle g,h \rangle = \langle a^{-1}ga, a^{-1}ha \rangle$ if $[g,h] = 1_G$.
\end{itemize}
By (P1) and (P2), if a subgroup $H \subset G$ commutes with $g$, the set $\langle g, H \rangle = \{ \langle g, h \rangle \mid h \in H \}$ is a subgroup of $H_2(G).$ If $G$ is commutative, the Pontryagin product induces a group homomorphism $\Lambda^2 G \to H_2(G)$, which is known (see, e.g., \cite[Theorem V.6.4]{bro}) to be isomorphic.
%---------------------------------------------------------------------------------------------------------------------
\section{Associated groups of quandles}\label{main-sec}
Let $X$ be a rack and let $\mathbb{Z}^{\oplus \mathcal{O}_X}$ denote the free abelian group generated by the connected components. Let $e_{[x]} \in \mathbb{Z}^{\oplus \mathcal{O}_X}$ denote the generator corresponding to the orbit through $x \in X$. It is well known (e.g., see \cite[Lemma 2.27]{nos1}) that a group homomorphism $\alpha: {\rm As}(X) \to \mathbb{Z}^{\oplus \mathcal{O}_X}$ is defined by $\alpha(e_x) = e_{[x]}$ and induces an isomorphism ${\rm As}(X)_{\rm ab} \cong \mathbb{Z}^{\oplus \mathcal{O}_X}$.
\par As recalled in the previous section, the action of ${\rm As}(X)$ on $X$ defines a group epimorphism $\psi: {\rm As}(X) \to {\rm Inn}(X).$ Here we should remark that the kernel of $\psi$ is contained in the center of ${\rm As}(X)$ since we have $g^{-1} e_x g = e_{\psi(g)(x)}$ for any $x \in X$ and $g \in {\rm As}(X)$.
\par Let $\overline{\rm Inn}(X)$ denote the image of $\alpha \times \psi: {\rm As}(X) \to \mathbb{Z}^{\oplus\mathcal{O}_X} \times {\rm Inn}(X)$, which acts on $X$ via the second projection. By considering the obvious equation $\alpha = p_1 \circ (\alpha \times \psi)$, where $p_1: \overline{\rm Inn}(X) \to \mathbb{Z}^{\oplus \mathcal{O}_X}$ is the first projection, we find that $p_1$ induces an isomorphism $\overline{\rm Inn}(X)_{\rm ab} \cong \mathbb{Z}^{\oplus \mathcal{O}_X}$. In particular, $H_1(\overline{\rm Inn}(X))$ is free and hence we have $H^2(\overline{\rm Inn}(X), A) \cong {\rm Hom}(H_2(\overline{\rm Inn}(X)), A)$ for any abelian group $A$ by the universal coefficient theorem. Thus, the central extensions of $\overline{\rm Inn}(X)$ by $A$ are classified by ${\rm Hom}(H_2(\overline{\rm Inn}(X)), A)$.

\begin{thm}\label{main-thm}
Let $X$ be a rack and let $A$ be the kernel of $\alpha \times \psi: {\rm As}(X) \to \mathbb{Z}^{\oplus \mathcal{O}_X} \times {\rm Inn}(X)$, where $\alpha: {\rm As}(X) \to \mathbb{Z}^{\oplus \mathcal{O}_X}$ is the abelianization and $\psi: {\rm As}(X) \to {\rm Inn}(X)$ is the canonical epimorphism. Then, there exists an isomorphism
\begin{equation}\label{main-eq}
A \cong H_2(\overline{{\rm Inn}}(X)) \;\bigg/ \sum_{x\in X} \langle \bar{s}_x, {\rm Stab}_{\overline{\rm Inn}(X)}(x) \rangle,
\end{equation}
where $\overline{\rm Inn} (X)$ is the image of $\alpha \times \psi$, $\bar{s}_x = (\alpha \times \psi)(e_x)$, $\langle -,- \rangle$ expresses the Pontryagin product, and $\sum_{x\in X} \langle \bar{s}_x, {\rm Stab}_{\overline{\rm Inn}(X)}(x) \rangle$ is the subgroup of $H_2(\overline{\rm Inn}(X))$ generated by the subgroups $\langle \bar{s}_x, {\rm Stab}_{\overline{\rm Inn}(X)}(x) \rangle$. More precisely, the central extension
\begin{equation}\label{main-seq}
1 \to A \to {\rm As}(X) \to \overline{\rm Inn}(X) \to 1
\end{equation}
is equivalent to the extension associated with $\varphi \in {\rm Hom}(H_2(\overline{\rm Inn}(X)), A) \cong H^2(\overline{\rm Inn}(X),A),$ where we identify $A$ with the quotient of $H_2(\overline{\rm Inn}(X))$ by the isomorphism {\rm(\ref{main-eq})} and $\varphi: H_2(\overline{\rm Inn}(X)) \to A$ is the projection.
\end{thm}

%\begin{proof}%[Proof of Theorem \ref{main-thm}]
\subsection{Proof of Theorem \ref{main-thm}}
In this section, we give a proof of Theorem \ref{main-thm}.
\par By the 5-term exact sequence of the central extension (\ref{main-seq}), we find that
$$H_2(\overline{\rm Inn}(X)) \to A \to {\rm As}(X)_{\rm ab} \to \overline{\rm Inn}(X)_{\rm ab} \to 0$$
is exact. By \cite[Theorem 4]{hil}, the leftmost map is equal to the evaluation by the cohomology class $\varphi \in H^2(\overline{\rm Inn}(X),A)$ ($\cong {\rm Hom}(H_2(\overline{\rm Inn}(X)),A)$) associated with the extension.
%this is a dual of the description (\cite[Theorem 4]{ho-se}) of $d_2^{0,1}$ in the cohomology spectral sequence, and can be checked by a concrete calculation.
Since ${\rm As}(X)_{\rm ab} \cong \overline{\rm Inn}(X)_{\rm ab} \cong \mathbb{Z}^{\oplus \mathcal{O}_X}$, this map is surjective. Thus, we may regard $A$ as a quotient of $H_2(\overline{\rm Inn}(X))$ and $\varphi: H_2(\overline{\rm Inn}(X)) \to A$ as the projection.
\par Let $A'$ be the right-hand side of (\ref{main-eq}) and consider the central extension
\begin{equation}\label{G'-seq}
1 \to A' \to G' \to \overline{\rm Inn}(X) \to 1
\end{equation}
associated with the cohomology class corresponding to the projection $\varphi': H_2(\overline{\rm Inn}(X)) \to A'$. First, we shall construct a group homomorphism $G' \to {\rm As}(X)$. Denote $H_2(\overline{\rm Inn}(X))$ by $\tilde{A}$ and let $\tilde{G}$ be the central extension of $\overline{\rm Inn}(X)$ associated with ${\rm id}_{\tilde{A}} \in H^2(\overline{\rm Inn}(X), \tilde{A})$; we have a short exact sequence
$$1 \to \tilde{A} \to \tilde{G} \to \overline{\rm Inn}(X) \to 1.$$
Since $\varphi: \tilde{A} \to A$ induces a homomorphism from $H^2(\overline{\rm Inn}(X),\tilde{A})$ to $H^2(\overline{\rm Inn}(X),A)$ that takes ${\rm id}_{\tilde{A}}$ to $\varphi$, there exists a group epimorphism $\tilde{p}: \tilde{G} \to {\rm As}(X)$.
\par To see that $\tilde{p}$ induces a homomorphism $p: G' \to {\rm As}(X)$, we prepare a lemma:

\begin{lem}\label{pp-lem}
Let $x \in X$ be an element and $f \in \overline{\rm Inn}(X)$ a stabilizer of $x$. Then, we have $\varphi(\langle \bar{s}_x, f \rangle) = 1_A$.
\end{lem}
\noindent Here, we use the multiplicative notation for $A$.
\begin{proof}
By the definition of $\overline{\rm Inn}(X)$, there exist elements $y_1, \dots, y_k \in X$ and signs $\epsilon_1, \dots, \epsilon_k \in \{ \pm 1\}$ such that $f = \bar{s}_{y_1}^{\epsilon_1} \cdots \bar{s}_{y_k}^{\epsilon_k}$. Let $x_i \in X$ denote $x \cdot \bar{s}_{y_1}^{\epsilon_1} \cdots \bar{s}_{y_i}^{\epsilon_i}$ for $i = 0, \dots, k$. Since $\bar{s}_{x_i} = (\bar{s}_{y_1}^{\epsilon_1} \cdots \bar{s}_{y_i}^{\epsilon_i})^{-1} \bar{s}_x (\bar{s}_{y_1}^{\epsilon_1} \cdots \bar{s}_{y_i}^{\epsilon_i})$, $\langle \bar{s}_x, f \rangle$ is represented by the inhomogeneous group cycle
$$\sum_{i = 0}^{k-1} c_i \in C_2(\overline{\rm Inn}(X)), \quad {\rm where} \quad c_i = [\bar{s}_{x_i}|\bar{s}_{y_{i+1}}^{\epsilon_{i+1}}] - [\bar{s}_{y_{i+1}}^{\epsilon_{i+1}} | \bar{s}_{x_{i+1}}].$$
\par We take a set-theoretic section $\sigma: \overline{\rm Inn}(X) \to {\rm As}(X)$ such that $\sigma(\bar{s}_{x}) = e_{x}$ and define a group cocycle $\phi \in C^2(\overline{\rm Inn}(X), A)$ representing $\varphi$ by $\sigma(g)\sigma(h) = \phi(g,h)\sigma(gh)$. Since $\bar{s}_{x_i}\bar{s}_{y_{i+1}}^{\epsilon_{i+1}} = \bar{s}_{y_{i+1}}^{\epsilon_{i+1}}\bar{s}_{x_{i+1}}$, we have
$$\phi(\bar{s}_{x_i},\bar{s}_{y_{i+1}}^{\epsilon_{i+1}})^{-1}\sigma(\bar{s}_{x_i})\sigma(\bar{s}_{y_{i+1}}^{\epsilon_{i+1}}) = \phi(\bar{s}_{y_{i+1}}^{\epsilon_{i+1}},\bar{s}_{x_{i+1}})^{-1}\sigma(\bar{s}_{y_{i+1}}^{\epsilon_{i+1}})\sigma(\bar{s}_{x_{i+1}})$$
and hence
$$\sigma(\bar{s}_{x_{i+1}}) = \phi(c_i)^{-1}\sigma(\bar{s}_{y_{i+1}}^{\epsilon_{i+1}})^{-1}\sigma(\bar{s}_{x_i})\sigma(\bar{s}_{y_{i+1}}^{\epsilon_{i+1}}) = \phi(c_i)^{-1}e_{y_{i+1}}^{-\epsilon_{i+1}}\sigma(\bar{s}_{x_i})e_{y_{i+1}}^{\epsilon_{i+1}}.$$
Therefore we have
\begin{align*}
e_x &= \sigma(\bar{s}_{x_k}) = \phi(c_0)^{-1} \cdots \phi(c_{k-1})^{-1} e_{y_k}^{-\epsilon_k} \cdots e_{y_1}^{-\epsilon_1} \sigma(\bar{s}_{x_0}) e_{y_1}^{\epsilon_1} \cdots e_{y_k}^{\epsilon_k}\\
&= \phi \left( \sum_{i = 0}^{k-1} c_i \right)^{-1} e_{y_k}^{-\epsilon_k} \cdots e_{y_1}^{-\epsilon_1} e_x e_{y_1}^{\epsilon_1} \cdots e_{y_k}^{\epsilon_k} = \varphi(\langle \bar{s}_x, f \rangle)^{-1} e_{x \cdot f},
\end{align*}
which implies $\varphi(\langle \bar{s}_x, f \rangle) = 1_A$ since $x \cdot f = x$.
\end{proof}

By Lemma \ref{pp-lem}, the projection $\varphi: \tilde{A} \to A$ factors through $A'$, i.e., there exists a group homomorphism $p_0: A' \to A$ such that $\varphi = p_0 \circ \varphi'$. Since $\varphi'$ and $\varphi$, regarded as cohomology classes, defines the group extensions (\ref{G'-seq}) and (\ref{main-seq}) respectively, there exists a surjective homomorphism $p:G' \to {\rm As}(X)$ such that the following diagram commutes:
$$\xymatrix{1 \ar[r] & A' \ar[r] \ar[d]^-{p_0} & G' \ar[r] \ar[d]^-{p} & \overline{\rm Inn}(X) \ar[r] \ar@{=}[d] & 1 \\
1 \ar[r] & A \ar[r] & {\rm As}(X) \ar[r] & \overline{\rm Inn}(X) \ar[r] & 1}.$$
\par In order to define the inverse $q: {\rm As}(X) \to G'$ of $p$, let $F(X)$ denote the free group generated by the symbols $e_x$ with $x \in X$ and define a group homomorphism $q_0: F(X) \to G'$ as follows. Let $E \subset X$ be a representative set of $\mathcal{O}_X$. For $x_0 \in E$, we arbitrarily take an element from $p^{-1}(e_{x_0}) \subset G'$ and let it be $q_0(e_{x_0})$. Next, for any $x \in X$, let $x_0 \in E$ be the representative of the orbit $[x] \in \mathcal{O}_X$ through $x$ and take any $f \in \overline{\rm Inn}(X)$ such that $x = x_0 \cdot f$; then we define $q_0(e_x) = \tilde{f}^{-1} q_0(e_{x_0}) \tilde{f}$, where $\tilde{f} \in G'$ is a lift of $f$.
\par Since $G' \to \overline{\rm Inn}(X)$ is a central extension, the definition of $q_0(e_x)$ does not depend on the choice of $\tilde{f}$. Moreover, we claim that it is independent of $f$. To see this, it is sufficient to check the independence in the case where $x = x_0$, i.e., we only have to show $q_0(e_{x_0}) = \tilde{f}^{-1} q_0(e_{x_0}) \tilde{f}$ for any $f \in {\rm Stab}_{\overline{\rm Inn}(X)}(x_0)$. Here, as in the proof of Lemma \ref{pp-lem}, we can verify that $\tilde{f}^{-1} q_0(e_{x_0}) \tilde{f} = \varphi'(\langle \bar{s}_{x_0}, f \rangle) q_0(e_{x_0})$. By the definition of $A'$, we find $\varphi'(\langle \bar{s}_{x_0}, f \rangle) = 1_{A'}$ and $\tilde{f}^{-1} q_0(e_{x_0}) \tilde{f} = q_0(e_{x_0})$, as claimed.
\par To see that $q_0$ induces a homomorphism $q: {\rm As}(X) \to G'$, we take $x,y \in X$ and express $x$ as $x_0 \cdot f$ with $x_0 \in E$ and $f \in \overline{\rm Inn}(X)$. Since $x* y = x_0\cdot f\bar{s}_y$ and $q_0(e_y)$ is a lift of $\bar{s}_y$, we have
$$q_0(e_{x*y}) = q_0(e_y)^{-1} \tilde{f}^{-1} q_0(e_{x_0}) \tilde{f} q_0(e_y) = q_0(e_y)^{-1} q_0(e_x) q_0(e_y),$$
where $\tilde{f}$ is a lift of $f$. Thus, the map $q_0$ canonically induces a homomorphism $q: {\rm As}(X) \to G'$.

\begin{proof}[Proof of Theorem \ref{main-thm}]
Let us check that $q \circ p$ is the identity, which verifies the injectivity of $p$ and hence completes the proof. For $x_0 \in E$, we have $q \circ p (q(e_{x_0})) = q(e_{x_0})$ by the definition. Moreover, if $q \circ p$ fixes an element $g \in G'$, it does any conjugate of $g$; this is because $q \circ p$ is compatible with the central extension $G' \to \overline{\rm Inn}(X)$. By the definition of $q$, for any $x \in X$, $q(e_x)$ is a conjugate of $q(e_{x_0})$ for some $x_0 \in E$, and hence is fixed by $q \circ p$. This implies that $q \circ p = {\rm id}_{G'}$ because the subgroup $G''$ of $G'$ generated by the elements $q(e_x)$ with $x \in X$ is equal to $G'$ by Lemma \ref{g'-lem} below.
\end{proof}

\begin{lem}\label{g'-lem}
Let $G'' \subset G'$ be a subgroup. If the composite $G'' \hookrightarrow G' \to \overline{\rm Inn}(X)$ is surjective, $G'' = G'$.
\end{lem}
\begin{proof}
The assumption of the surjectivity implies that $G''$ is a central extension of $\overline{\rm Inn}(X)$ by a subgroup $A''$ of $A'$. Let $\varphi'' \in H^2(\overline{\rm Inn}(X), A'')$ be the cohomology class associated with the central extension $G'' \to \overline{\rm Inn}(X)$. By the functoriality of the extension classes, the homomorphism induced by the inclusion map $A'' \to A'$ between the coefficient groups sends $\varphi'' \in H^2(\overline{\rm Inn}(X), A'')$ to $\varphi' \in H^2(\overline{\rm Inn}(X), A')$, which implies that $A'' = A'$ since the evaluation by $\varphi'$ is surjective. Thus, we have $G'' = G'$, as required.
\end{proof}

\subsection{Preparation for applications}
Before proceeding, we prepare notations and basic tools that are helpful in applying Theorem \ref{main-thm}.
%\par In this paper, for a $\mathbb{Z}[T^{\pm 1}]$-module $M$, or equivalently for an abelian group $M$ and an automorphism $T$ on $M$, we denote $M/(1-T)M$ by $M_T$ and ${\rm Ker}\,(1-T) \subset M$ by $M^T$.
\par Let $X$ be a quandle and define $\psi: {\rm As}(X) \to {\rm Inn}(X)$, $\alpha: {\rm As}(X) \to \mathbb{Z}^{\oplus \mathcal{O}_X}$, $\overline{\rm Inn}(X) \subset \mathbb{Z}^{\oplus\mathcal{O}_X} \times {\rm Inn}(X)$, and $\bar{s}_x \in \overline{\rm Inn}(X)$ for $x \in X$ as in Theorem \ref{main-thm}. Let $\epsilon_0: \mathbb{Z}^{\oplus \mathcal{O}_X} \to \mathbb{Z}$ be the group homomorphism that sends $e_{[x]} \in \mathbb{Z}^{\oplus \mathcal{O}_X}$ to $1 \in \mathbb{Z}$, and let $\bar{\epsilon}: \overline{\rm Inn}(X) \to \mathbb{Z}$ be the composite of the first projection and $\epsilon_0$. Then, we define $\overline{\rm Inn}_0(X)$ to be ${\rm Ker}\,\bar{\epsilon}$ and let ${\rm Inn}_0(X) \subset {\rm Inn}(X)$ denote $p_2(\overline{\rm Inn}_0(X))$, where $p_2: \mathbb{Z}^{\oplus\mathcal{O}_X} \times {\rm Inn}(X) \to {\rm Inn}(X)$ is the second projection. If $X$ is connected, we can identify $\overline{\rm Inn}_0(X)$ with ${\rm Inn}_0(X)$ since $\overline{\rm Inn}_0(X) = \{0\} \times {\rm Inn}_0(X)$.
\par Since $\bar{\epsilon}$ is surjective, we have a group extension
$$1 \to \overline{\rm Inn}_0(X) \to \overline{\rm Inn}(X) \to \mathbb{Z} \to 1.$$
The Lyndon-Hochschild-Serre spectral sequence for this extension yields short exact sequences
\begin{equation}\label{ss-ses}
0 \to H_n(\overline{\rm Inn}_0(X))_\sigma \to H_n(\overline{\rm Inn}(X)) \to H_{n-1}(\overline{\rm Inn}_0(X))^\sigma \to 0
\end{equation}
for $n \geq 0$, where $\sigma$ is the conjugation by $\bar{s}_{x_0}$ for any fixed point $x_0 \in X$. In particular, the sequence (\ref{ss-ses}) in the case $n = 2$ is useful in the calculation of $A$ in Theorem \ref{main-thm}.
\par Also, the sum of Pontryagin products in the equation (\ref{main-eq}) can be simplified:

\begin{lem}\label{simp-lem}
Let $X$ be a quandle and take any representative set $E \subset X$ of $\mathcal{O}_X$. Then we have
$$\sum_{x\in X} \langle \bar{s}_x, {\rm Stab}_{\overline{\rm Inn}(X)}(x) \rangle = \sum_{x_0 \in E} \langle \bar{s}_{x_0}, {\rm Stab}_{\overline{\rm Inn}_0(X)}(x_0) \rangle.$$
\end{lem}
\begin{proof}
For $x \in X$ and $f \in \overline{\rm Inn}(X)$, $\bar{s}_{x\cdot f} = f^{-1} \bar{s}_x f$ and ${\rm Stab}_{\overline{\rm Inn}(X)}(x\cdot f) = f^{-1}{\rm Stab}_{\overline{\rm Inn}(X)}(x)f$. Therefore the property (P3) recalled in Section \ref{pre-sec} shows that the range of the sum can be replaced with the representatives $E$, instead of the whole set $X$. Furthermore, since ${\rm Stab}_{\overline{\rm Inn}(X)}(x)$ is generated by ${\rm Stab}_{\overline{\rm Inn}_0(X)}(x)$ and $\bar{s}_x$, the properties (P1) and (P2) lead to the required equation.
\end{proof}
%---------------------------------------------------------------------------------------------------------------------
\section{Applications}\label{ap-sec}
In this section, we apply Theorem \ref{main-thm} to specific quandles. We give formulae of the associated groups and the second quandle homology groups for Alexander quandles in Section \ref{alex-sec} and for connected generalized Alexander quandles in Section \ref{ga-sec}. We also consider the product of connected quandles in Section \ref{prod-sec}, and show a K$\ddot{\rm u}$nneth-like theorem for the second homology.
%---------------------------------------------------------------------------------------------------------------------
\subsection{Alexander quandles}\label{alex-sec}
Let $M$ be a $\mathbb{Z}[T^{\pm 1}]$-module; that is, $M$ is an abelian group and $T$ is an automorphism of $M$. The set $M$ with the binary operation $*$ given by $x * y = Tx + (1-T)y$ is a quandle and called the \textit{Alexander quandle} of the pair $(M,T)$. In this paper, we denote this quandle by $Q_{M,T}$.
\par It is known (\cite{li-ne}) that the connected component of an Alexander quandle $X = Q_{M,T}$ that contains an element $x$ of $X$ is equal to $x + (1-T)M$, and hence we can identify $\mathcal{O}_X$ with $M_T$; especially, $Q_{M,T}$ is connected if and only if $1-T$ is surjective.

\begin{thm}[Theorem \ref{alex-thm0}]\label{alex-thm}
Let $M$ be a $\mathbb{Z}[T^{\pm 1}]$-module and $X = Q_{M,T}$ the Alexander quandle. Then, the kernel $A$ in Theorem \ref{main-thm} is isomorphic to
$$(\Lambda^2 (1-T)M)_T \oplus ((1-T)M)^T,$$
where $T$ acts on the exterior product diagonally, and we have
$$H_2^Q(X) \cong \left( (\Lambda^2 (1-T)M)_T \oplus ((1-T)M)^T \oplus \mathbb{Z}^{\oplus (|M_T| -1)} \right)^{\oplus M_T}.$$
\end{thm}
%\begin{thm}\label{alex-thm}
%Let $M$ be a $\mathbb{Z}[T^{\pm 1}]$-module and $X = Q_{M,T}$ the Alexander quandle. Then, the kernel $A$ in Theorem \ref{main-thm} is isomorphic to
%$$(\Lambda^2 (1-T)M)_T \oplus ((1-T)M)^T,$$
%where $T^{\pm 1}$ acts on the exterior product diagonally. Furthermore, define a group homomorphism $\xi: \mathbb{Z}^{\oplus M_T} \to (1-T)M / (1-T)^2M$ by $\xi(e_x) = (1-T)x$ and set $\bar{\xi} = \epsilon_0 \oplus \xi: \mathbb{Z}^{\oplus M_T} \to \mathbb{Z} \oplus (1-T)M / (1-T)^2M$. Then we have $H_2^Q(X) \cong (A \oplus {\rm Ker}\,\bar{\xi})^{\oplus M_T}.$ In fact, ${\rm Ker}\,\bar{\xi}$ is isomorphic to $\mathbb{Z}^{\oplus (|M_T| -1)}$ as an abelian group, and therefore
%$$H_2^Q(X) \cong \left( (\Lambda^2 (1-T)M)_T \oplus ((1-T)M)^T \oplus \mathbb{Z}^{\oplus (|M_T| -1)} \right)^{\oplus M_T}.$$
%\end{thm}
\noindent Here, $|M_T|$ denotes the cardinality of $M_T$ and we do not assume $|M_T| < \infty$. We should note that $|M_T| = |M_T| - 1$ if $M_T$ has infinite order.

\begin{proof}
In order to apply Theorem \ref{main-thm}, we first compute $H_2(\overline{\rm Inn}(X))$. Since the inner automorphism $s_x$ is an affine transformation $T \bullet + (1-T)x$, we find that if $a \in \mathbb{Z}^{\oplus M_T}$ and $f \in {\rm Inn}(X)$ satisfy $(a,f) \in \overline{\rm Inn}(X)$, then $f$ is an affine transformation $T^{\epsilon_0(a)}\bullet + y$ for some $y \in (1-T)M$; in particular, we can regard $\overline{\rm Inn}_0(X)$ as a subgroup of the abelian group $\epsilon_0^{-1}(0) \times (1-T)M$. Recalling that the Pontryagin product induces an isomorphism $\Lambda^2 G \to H_2(G)$ for any abelian group $G$, we have $H_2(\overline{\rm Inn}_0(X)) \cong \Lambda^2 \overline{\rm Inn}_0(X)$.
\par Taking $0 \in X$ as $x_0$ in (\ref{ss-ses}), we obtain an exact sequence
\begin{equation}\label{alex-seq}
0 \to H_2(\overline{\rm Inn}_0(X))_T \to H_2(\overline{\rm Inn}(X)) \to H_1(\overline{\rm Inn}_0(X))^T \to 0.
\end{equation}
Considering the commutative subgroup $H \subset \overline{\rm Inn}(X)$ generated by $\bar{s}_0$ and $\overline{\rm Inn}_0(X)^T$, we have a trivial extension $\overline{\rm Inn}_0(X)^T \hookrightarrow H \twoheadrightarrow \mathbb{Z}$ and an exact sequence
\begin{equation}\label{sp-seq}
0 \to H_2(\overline{\rm Inn}_0(X)^T) \to H_2(H) \to H_1(\overline{\rm Inn}_0(X)^T) \to 0.
\end{equation}
Since the injective homomorphism $H_2(\overline{\rm Inn}_0(X)^T) \to H_2(H)$ is induced by the inclusion, the epimorphism $r: H \to \overline{\rm Inn}_0(X)^T$ defined by $r(a, f) = \bar{s}_0^{-\epsilon_0(a)}(a, f)$ induces a splitting of (\ref{sp-seq}). Let $s: H_1(\overline{\rm Inn}_0(X)^T) \to H_2(H)$ be the section of this splitting. Since the exact sequences (\ref{sp-seq}) and (\ref{alex-seq}) are compatible with the inclusion homomorphisms, $\iota_* \circ s: H_1(\overline{\rm Inn}_0(X)^T) \to H_2(\overline{\rm Inn}(X))$ gives a splitting of (\ref{alex-seq}), where $\iota: H \to \overline{\rm Inn}(X)$ is the inclusion. Hence,
\begin{eqnarray}
H_2(\overline{\rm Inn}(X)) &\cong& H_2(\overline{\rm Inn}_0(X))_T \oplus H_1(\overline{\rm Inn}_0(X))^T \notag\\
&\cong& (\Lambda^2 \overline{\rm Inn}_0(X))_T \oplus \overline{\rm Inn}_0(X)^T. \label{alex-eq}
\end{eqnarray}
Here, we remark that the image of the Pontryagin product $\langle \bar{s}_0,- \rangle: \overline{\rm Inn}_0(X)^T \to H_2(H)$ is contained in ${\rm Ker}\,r_*$ since $r(\bar{s}_0)  = (0, {\rm id}_X)$. Moreover, $\langle \bar{s}_0,- \rangle$ gives an isomorphism $\overline{\rm Inn}_0(X)^T \cong {\rm Ker}\,r_*$ because
$$H_2(H) \cong \Lambda^2H = \Lambda^2 \overline{\rm Inn}_0(X)^T \oplus \langle \bar{s}_0 \rangle \otimes \overline{\rm Inn}_0(X)^T \cong H_2(\overline{\rm Inn}_0(X)^T) \oplus \overline{\rm Inn}_0(X)^T,$$
where the isomorphisms are given by the Pontryagin products. Therefore we may consider that the injection from $\overline{\rm Inn}_0(X)^T$ to $H_2(\overline{\rm Inn}(X))$ in the isomorphism (\ref{alex-eq}) is induced by the Pontryagin product $\langle \bar{s}_0, - \rangle$.%(although $\langle \bar{s}_0, - \rangle$ coincides with the section $s$, we do not need this fact).
%Here, the Pontryagin product $\langle \bar{s}_0,- \rangle: H_1(\overline{\rm Inn}_0(X)^T) \to H_2(H)$ gives a splitting of (\ref{sp-seq}); this can be seen, e.g., by taking $S^1 \times K(\overline{\rm Inn}_0(X)^T, 1)$ as a $K(H, 1)$-space and considering the trivial fibration $S^1 \times K(\overline{\rm Inn}_0(X)^T, 1) \to S^1$ to obtain (\ref{sp-seq}). Since the exact sequences (\ref{sp-seq}) and (\ref{alex-seq}) are compatible with the inclusion homomorphisms, the Pontryagin product $\langle \bar{s}_0, -\rangle$ also gives a splitting of (\ref{alex-seq}). Thus, we have
%\begin{eqnarray}
%H_2(\overline{\rm Inn}(X)) &\cong& H_2(\overline{\rm Inn}_0(X))_T \oplus H_1(\overline{\rm Inn}_0(X))^T \notag\\
%&\cong& (\Lambda^2 \overline{\rm Inn}_0(X))_T \oplus \overline{\rm Inn}_0(X)^T. \label{alex-eq}
%\end{eqnarray}
%Since the isomorphism
%$$H_1(\overline{\rm Inn}_0(X))^T \cong \overline{\rm Inn}_0(X)^T \to H_2(\overline{\rm Inn}(X))/H_2(\overline{\rm Inn}_0(X))_T$$
%is (at least up to sign) equal to the one induced by the Pontryagin product $\langle \bar{s}_0, - \rangle$ with $\bar{s}_0$, the map $\langle \bar{s}_0, - \rangle: \overline{\rm Inn}_0(X)^T \to H_2(\overline{\rm Inn}(X))$ gives a splitting of the exact sequence (\ref{alex-seq}). Thus, we have
%\begin{eqnarray}
%H_2(\overline{\rm Inn}(X)) &\cong& H_2(\overline{\rm Inn}_0(X))_T \oplus H_1(\overline{\rm Inn}_0(X))^T \notag\\
%&\cong& (\Lambda^2 \overline{\rm Inn}_0(X))_T \oplus \overline{\rm Inn}_0(X)^T. \label{alex-eq}
%\end{eqnarray}
\par We use $\overline{\rm Inn}_0(X)$ to compute $A$, as in Lemma \ref{simp-lem}. Since $\overline{\rm Inn}_0(X)$ acts on $X$ as translations, the isotropy subgroups are calculated as
$${\rm Stab}_{\overline{\rm Inn}_0(X)}(x) = \overline{\rm Inn}_0(X) \cap (\mathbb{Z}^{\oplus M_T} \times \{ {\rm id}_X \})$$
for any $x \in X$. Denote this isotropy subgroup by $I$, which does not depend on $x$ and hence is contained in the center of $\overline{\rm Inn}(X)$. By the property (P2) of Pontyagin products, we find
$$\sum_{x \in X} \langle \bar{s}_x, {\rm Stab}_{\overline{\rm Inn}_0(X)}(x) \rangle = \langle \overline{\rm Inn}(X), I \rangle = \langle \overline{\rm Inn}_0(X), I \rangle + \langle \bar{s}_0, I \rangle.$$
Recalling that the isomorphism (\ref{alex-eq}) was given by Pontryagin products, we find these two terms respectively contained in the two summands of (\ref{alex-eq}). The projection $\overline{\rm Inn}_0(X) \to {\rm Inn}_0(X)$ induces an isomorphism $\overline{\rm Inn}_0(X)/I \cong (1-T)M$ and hence we have
\begin{eqnarray*}
A &\cong& H_2(\overline{\rm Inn}(X))/\langle \overline{\rm Inn}(X), I \rangle\\[2pt]
&=& \frac{\langle \overline{\rm Inn}_0(X), \overline{\rm Inn}_0(X) \rangle}{\langle \overline{\rm Inn}_0(X), I \rangle} \oplus \frac{\langle \bar{s}_0, \overline{\rm Inn}_0(X)^T\rangle}{\langle \bar{s}_0, I \rangle}\\[2pt]
&\cong& (\Lambda^2(1-T)M)_T \oplus ((1-T)M)^T
\end{eqnarray*}
by Theorem \ref{main-thm} and Lemma \ref{simp-lem}, as required.
\par Let us compute the homology group $H_2^Q(X)$. By Theorem \ref{eis-thm}, we have
$$H_2^Q(X) \cong \bigoplus_{x \in E} ({\rm Stab}_{{\rm As}(X)}(x) \cap {\rm Ker}\,\epsilon)_{\rm ab}.$$
Let $\tilde{I}$ denote ${\rm Stab}_{{\rm As}(X)}(x) \cap {\rm Ker}\,\epsilon$. We remark that $A \hookrightarrow \tilde{I} \twoheadrightarrow I$ is a central extension. In fact, this extension is trivial: Since $I$ is a subgroup of the free abelian group $\mathbb{Z}^{\oplus M_T} \times \{{\rm id}_X\}$, $I$ is free abelian and hence $H^2(I, A) \cong {\rm Hom}(H_2(I), A)$ by the universal coefficient theorem. Furthermore, the second homology of an abelian group is generated by Pontryagin products, but a Pontryagin product in $I$ vanishes in $A$ as seen above. Thus, the restriction of the cohomology class of the extension ${\rm As}(X) \to \overline{\rm Inn}(X)$ to $I$ vanishes by the naturality of the universal coefficient theorem and we have $\tilde{I} \cong A \oplus I$.
\par Define a group homomorphism $\xi: \mathbb{Z}^{\oplus M_T} \to (1-T)M / (1-T)^2M$ by $\xi(e_x) = (1-T)x$. Here, we claim that $I$ is isomorphic to the kernel of $\bar{\xi} = \epsilon_0 \oplus \xi: \mathbb{Z}^{\oplus M_T} \to \mathbb{Z} \oplus (1-T)M / (1-T)^2M$. Since $x\cdot e_y = Tx+ (1-T)y$, we find that $0 \cdot f \in \xi(a)$ holds for $(a,f) \in \overline{\rm Inn}_0(X)$. If $a \in {\rm Ker}\,\bar{\xi}$, i.e., if $\epsilon_0(a) = 0$ and $0 \cdot f = (1-T)^2x$ for some $x \in M$, then we have $\bar{s}_0 \bar{s}_{(1-T)x}^{-1} (a,f) = (a,{\rm id}_X)$ and hence $(a, {\rm id}_X) \in I$. Thus, $I \cong {\rm Ker}\,\bar{\xi}$ and $\tilde{I} \cong A \oplus {\rm Ker}\, \bar{\xi}$; we obtain $H_2^Q(X) \cong (A \oplus {\rm Ker}\,\bar{\xi})^{\oplus M_T}$.
\par As the kernel ${\rm Ker}\,\bar{\xi}$ is a subgroup of the free abelian group $\mathbb{Z}^{\oplus M_T}$, it is also free. To see the rank, we consider the $\mathbb{Q}$-tensored map $\bar{\xi}_{\mathbb{Q}}: \mathbb{Q}^{\oplus M_T} \to \mathbb{Q} \oplus (\mathbb{Q} \otimes ((1-T)M)_T)$. In $\mathbb{Q}^{\oplus M_T}$, we define $e'_x = e_x - e_0$ for torsion elements $x \in M_T \backslash \{ 0 \}$, and $e'_x = e_x - (1/2)e_{2x} - (1/2)e_0$ for the other nonzero elements. The elements $e'_x$ are contained in ${\rm Ker}\, \bar{\xi}_{\mathbb{Q}}$, and a brief consideration shows that these are linearly independent. Therefore, since $\epsilon_0$ is surjective, we have
$$|M_T| - 1 \geq {\rm rank}\,({\rm Ker}\,\bar{\xi}) = {\rm dim}\,({\rm Ker}\,\bar{\xi}_{\mathbb{Q}}) \geq |\{e'_x \mid x \in M_T\backslash\{0\}\}| = |M_T| - 1,$$
which concludes that ${\rm Ker}\, \bar{\xi} \cong \mathbb{Z}^{\oplus (|M_T|-1)}.$
\end{proof}

\begin{rem}
The second homology group of a latin Alexander quandle $Q_{M,T}$ has already been computed in \cite{cla} and \cite{bimnp}. Here, a quandle is \textit{latin} if the left-multiplication maps $x * \bullet$ are bijective, and we can check that an Alexander quandle $Q_{M, T}$ is latin if and only if $1 - T$ is invertible. If $Q_{M,T}$ is latin, Theorem \ref{alex-thm} shows that $H_2^Q(Q_{M,T}) \cong (\Lambda^2 M)_T$, which is identical to the description in \cite[Corollary 4.3]{bimnp}.
\end{rem}
%-------------------------------------------------------------------------------------------------------------
\subsection{Connected generalized Alexander quandles}\label{ga-sec}
Let $G$ be a group and $\phi$ an automorphism of $G$. We define a binary oparation $*$ on $G$ by $x * y = \phi(xy^{-1}) y$ for $x, y \in G$; the pair $(G,*)$ is a quandle. We call this quandle a \textit{generalized Alexander quandle} and denote it by $Q_{G,\phi}$.
\par For a generalized Alexander quandle $Q_{G,\phi}$, consider the subgroup $H \subset G$ generated by the elements $\phi(g^{-1})g$ with all $g \in G$. By the definition of the binary operation, we can easily find that for any $f \in {\rm Inn}_0(Q_{G, \phi})$, there exists $h \in H$ such that $f(x) = xh$ for any $x \in Q_{G,\phi}$; remark that $h = f(1_G)$. Since $s_{1_G}^{-1}s_g \in {\rm Inn}_0(Q_{G,\phi})$ maps $1_G \in Q_{G,\phi}$ to $\phi(g^{-1})g$, the map sending $f$ to $f(1_G)$ gives an isomorphism ${\rm Inn}_0(Q_{G,\phi}) \cong H$. Thus, each connected component $[x] \in \mathcal{O}_{Q_{G,\phi}}$ through $x \in Q_{G,\phi}$ is equal to the right coset $xH \in G/H$ and hence there is a bijection between $\mathcal{O}_{Q_{G,\phi}}$ and $G/H$; in particular, $Q_{G,\phi}$ is connected if and only if $G = H$.

\begin{thm}[Theorem \ref{ga-thm0}]
Let $G$ be a group and $\phi$ an automorphism of $G$. If the Alexander quandle $Q_{G, \phi}$ is connected and the endomorphism $1- \phi$ on $G_{\rm ab}$ is invertible, then the kernel $A$ in Theorem \ref{main-thm} is isomorphic to $H_2(G)_{\phi}$. In particular, we have
$$H_2^Q(Q_{G,\phi}) \cong H_2(G)_{\phi}.$$
\end{thm}
\begin{proof}
As seen above, the group ${\rm Inn}_0(Q_{G,\phi})$ can be identified with $G$, acting on $Q_{G,\phi}$ as the right multiplication in the group, since $Q_{G,\phi}$ is connected. Since the endomorphism $1 - \phi$ on $H_1(G)$ is assumed to be invertible, the exact sequence (\ref{ss-ses}) implies that $H_2(G)_\phi \cong H_2(\overline{\rm Inn}(Q_{G,\phi}))$. By Theorem \ref{main-thm} and Lemma \ref{simp-lem}, we find the required isomorphism $A \cong H_2(G)_\phi$, since the action of ${\rm Inn}_0(Q_{G,\phi})$ on $Q_{G,\phi}$ is free.
\par The calculation of $H_2^Q(Q_{G,\phi})$ follows from Theorem \ref{eis-thm}, immediately.
\end{proof}

\begin{rem}
For a connected generalized Alexander quandle $Q_{G,\phi}$, the endomorphism $1 - \phi \in {\rm End}(G_{\rm ab})$ is surjective. Thus, the assumption of the invertibility is automatically satisfied for a rather wide class of groups; e.g., for finitely generated groups.
\par Also, even if $1 - \phi$ is not invertible, we have an exact sequence
$$0 \to H_2(G)_{\phi} \to H_2^Q(Q_{G,\phi}) \to (G_{\rm ab})^\phi \to 0$$
if $Q_{G, \phi}$ is connected, as in the proof above. If there exists a section $(G_{\rm ab})^\phi \to G$ of the abelianization map $G \to G_{\rm ab}$, this sequence splits as in the proof of Theorem \ref{alex-thm}.
\end{rem}
%-------------------------------------------------------------------------------------------------------------
\subsection{Products of connected quandles}\label{prod-sec}
If $(X, *_X)$ and $(Y,*_Y)$ are quandles, we can define the \textit{product} $(X \times Y,*)$ of them; the quandle operation is given by $(x,y) * (x',y') = (x *_X x', y *_Y y')$. In this section, we consider the product of connected quandles and deduce a ``K$\ddot{\rm u}$nneth theorem" for the associated group and the second homology.

\begin{thm}[Theorem \ref{prod-thm0}]
Let $X,Y$ be connected quandles. Fixing any points $x_0 \in X, y_0 \in Y$, we define an automorphism $\sigma$ of ${\rm Inn}_0(X) \times {\rm Inn}_0(Y)$ by $\sigma(f, g) = (s_{x_0}^{-1}fs_{x_0}, s_{y_0}^{-1} g s_{y_0})$. If $Z$ is one of $X, Y, X \times Y$, let $A_Z$ be the kernel of $\alpha_Z \times \psi_Z: {\rm As}(Z) \to \mathbb{Z} \times {\rm Inn}(Z)$, as in Theorem {\rm \ref{main-thm}}. Then, we have
\begin{eqnarray*}
A_{X \times Y} &\cong& A_X \oplus A_Y \oplus ({\rm Inn}_0(X)_{\rm ab} \otimes {\rm Inn}_0(Y)_{\rm ab})_\sigma,\\
H_2^Q(X\times Y) &\cong& H_2^Q(X) \oplus H_2^Q(Y) \oplus ({\rm Inn}_0(X)_{\rm ab} \otimes {\rm Inn}_0(Y)_{\rm ab})_\sigma.
\end{eqnarray*}
\end{thm}
\begin{proof}
First, we check that ${\rm Inn}_0(X \times Y) \cong {\rm Inn}_0(X) \times {\rm Inn}_0(Y)$. There exist obvious projections from $\overline{\rm Inn}(X \times Y)$ onto $\overline{\rm Inn}(X)$ and $\overline{\rm Inn}(Y)$ and these maps preserve the first component $\mathbb{Z}$. Hence we have a group homomorphism $\eta: {\rm Inn}_0(X \times Y) \to {\rm Inn}_0(X) \times {\rm Inn}_0(Y)$. Since we can regard the both sides as subgroups of ${\rm Aut}(X\times Y)$, $\eta$ is injective. For any $f \in {\rm Inn}_0(X)$, we can take $x_i \in X$ and $\epsilon_i = \pm1$ with $\sum_i \epsilon_i = 0$ such that $f = s_{x_1}^{\epsilon_1} \cdots s_{x_k}^{\epsilon_k}$, and then we find
$$\eta\bigl(\bar{s}_{(x_1,y_0)}^{\epsilon_1} \cdots \bar{s}_{(x_k,y_0)}^{\epsilon_k}\bigr) = (f, {\rm id}_Y),$$
which implies that the image of $\eta$ contains ${\rm Inn}_0(X) \times \{{\rm id}_Y\}.$ In the same way we find ${\rm Im}\,\eta \supset \{{\rm id}_X\} \times {\rm Inn}_0(Y)$, and hence $\eta$ is surjective.
\par By applying (\ref{ss-ses}) to $X, Y, X \times Y$, we have exact sequences; we here denote them by ${\rm (\ref{ss-ses})}_X, {\rm (\ref{ss-ses})}_Y, {\rm (\ref{ss-ses})}_{X \times Y}$, respectively. Let $B$ denote $({\rm Inn}_0(X)_{\rm ab} \otimes {\rm Inn}_0(Y)_{\rm ab})_\sigma$. By the K$\ddot{\rm u}$nneth theorem, the first and third columns in the following diagram are isomorphisms:
$${\footnotesize\xymatrix@C=10pt{
H_2({\rm Inn}_0(X))_\sigma \oplus H_2({\rm Inn}_0(Y))_\sigma \oplus B \ar@{^{(}->}[r] \ar[d]^-{\cong} & H_2(\overline{\rm Inn}(X)) \oplus H_2(\overline{\rm Inn}(Y)) \oplus B \ar@{->>}[r] \ar@{.>}[d]^-{\iota} & H_1({\rm Inn}_0(X))^\sigma \oplus H_1({\rm Inn}_0(Y))^\sigma \ar[d]^-{\cong}\\
H_2({\rm Inn}_0(X\times Y))_\sigma \ar@{^{(}->}[r] & H_2(\overline{\rm Inn}(X \times Y)) \ar@{->>}[r] & H_1({\rm Inn}_0(X \times Y))^\sigma.
}}$$
Here, the first row is exact by ${\rm (\ref{ss-ses})}_X$ and ${\rm (\ref{ss-ses})}_Y$, and so is the second by ${\rm (\ref{ss-ses})}_{X \times Y}$. Define $\iota_X: \overline{\rm Inn}(X) \to \overline{\rm Inn}(X \times Y)$ and $\iota_Y: \overline{\rm Inn}(Y) \to \overline{\rm Inn}(X \times Y)$ by $\iota_X(a, f) = (a, f, s_{y_0}^a)$ and $\iota_Y(a, g) = (a, s_{x_0}^a, g)$, and let $\iota_B: B \to H_2(\overline{\rm Inn}(X \times Y))$ be the composite of the cross product $\times: B \to H_2({\rm Inn}_0(X \times Y))_\sigma$ and the inclusion homomorphism. Denoting $\iota_{X*} \oplus \iota_{Y*} \oplus \iota_B$ by $\iota$, we find that the above diagram commutes. Thus, the five-lemma shows that $\iota$ is isomorphic.
%\par By applying (\ref{ss-ses}) to $Z = X, Y, X \times Y$, we have an exact sequence; we here denote it by ${\rm (\ref{ss-ses})}_Z$. Because ${\rm Inn}_0(X \times Y) \cong {\rm Inn}_0(X) \times {\rm Inn}_0(Y)$, the K$\ddot{\rm u}$nneth theorem decomposes ${\rm (\ref{ss-ses})}_{X \times Y}$ into ${\rm (\ref{ss-ses})}_X, {\rm (\ref{ss-ses})}_Y,$ and terms on $A_{XY} = ({\rm Inn}_0(X)_{\rm ab} \otimes {\rm Inn}_0(Y)_{\rm ab})_\sigma$: we have an exact sequence
%\begin{multline*}
%0 \to H_2({\rm Inn}_0(X))_\sigma \oplus H_2({\rm Inn}_0(Y))_\sigma \oplus A_{XY} \\
%\to H_2(\overline{\rm Inn}(X)) \oplus H_2(\overline{\rm Inn}(Y)) \oplus A_{XY} \\
%\to H_1({\rm Inn}_0(X))^\sigma \oplus H_1({\rm Inn}_0(Y))^\sigma \to 0
%\end{multline*}
%which is equivalent to ${\rm (\ref{ss-ses})}_{X \times Y}$. Thus, the second homology $H_2(\overline{\rm Inn}(X \times Y))$ is the direct sum of $H_2(\overline{\rm Inn}(X)), H_2(\overline{\rm Inn}(Y)),$ and $A_{XY}$.
\par When $(Z,z_0)$ is one of $(X,x_0),(Y,y_0), (X \times Y, (x_0,y_0))$, let $I_Z$ denote ${\rm Stab}_{\overline{\rm Inn}_0(Z)}(z_0)$. The kernel $A_{X \times Y}$ is the quotient of $H_2(\overline{\rm Inn}(X \times Y))$ by $\langle \bar{s}_{(x_0, y_0)}, I_{X\times Y} \rangle$. Because $I_{X \times Y} \cong I_X \times I_Y$, we have
$$\langle \bar{s}_{(x_0, y_0)}, I_{X \times Y} \rangle = \langle \bar{s}_{x_0}, I_X \rangle + \langle \bar{s}_{y_0}, I_Y \rangle,$$
where we regard $H_2(\overline{\rm Inn}(X))$ and $H_2(\overline{\rm Inn}(Y))$ as subgroups of $H_2(\overline{\rm Inn}(X \times Y))$ by the isomorphism $\iota$. Since $A_X \cong H_2(\overline{\rm Inn}(X))/\langle \bar{s}_{x_0}, I_X \rangle$ and $A_Y \cong H_2(\overline{\rm Inn}(Y))/\langle \bar{s}_{y_0}, I_Y \rangle$, we obtain the first isomorphism of the theorem.
\par By Theorem \ref{eis-thm}, the second quandle homology $H^Q_2(X \times Y)$ is isomorphic to the abelianization $(\tilde{I}_{X \times Y})_{\rm ab}$ of a subgroup $\tilde{I}_{X \times Y} = {\rm Stab}_{{\rm As}(X \times Y)}(x_0,y_0) \cap {\rm Ker}\, \epsilon$ of the isotropy group. We should remark that $A_{X \times Y} \hookrightarrow \tilde{I}_{X\times Y} \twoheadrightarrow I_{X \times Y}$ is a central extension. Let $f \in H^2(\overline{\rm Inn}(X \times Y), A_{X \times Y})$ be the cohomology class associated with the central extension $A_{X\times Y} \hookrightarrow {\rm As}(X \times Y) \twoheadrightarrow \overline{\rm Inn}(X \times Y)$. The extension $\tilde{I}_{X \times Y} \to I_{X \times Y}$ is given by the restriction of $f$ to $I_{X \times Y}$.
\par By Theorem \ref{main-thm}, $f$ corresponds to the projection from $H_2(\overline{\rm Inn}(X \times Y))$ onto $A_{X \times Y}$. We have
\begin{eqnarray}
&&H^2(\overline{\rm Inn}(X \times Y), A_{X \times Y})\notag\\
&\cong& {\rm Hom}(H_2(\overline{\rm Inn}(X)) \oplus H_2(\overline{\rm Inn}(Y)) \oplus B, A_{X \times Y})\notag\\
&\cong& H^2(\overline{\rm Inn}(X), A_{X \times Y}) \oplus H^2(\overline{\rm Inn}(Y), A_{X \times Y}) \oplus {\rm Hom}(B, A_{X \times Y}) \label{h2-eq}
\end{eqnarray}
and, as seen above, the projection of $H_2(\overline{\rm Inn}(X \times Y))$ onto $A_{X \times Y}$ is the direct sum of the two projections $f_X: H_2(\overline{\rm Inn}(X)) \to A_X$ and $f_Y: H_2(\overline{\rm Inn}(Y)) \to A_Y$ and the identity map ${\rm id}_B$. Thus, under the isomorphism (\ref{h2-eq}), $f$ is equal to the sum of $f_X \in H^2(\overline{\rm Inn}(X), A_X) \subset H^2(\overline{\rm Inn}(X), A_{X \times Y})$, $f_Y \in H^2(\overline{\rm Inn}(Y), A_Y) \subset H^2(\overline{\rm Inn}(Y), A_{X \times Y})$, and ${\rm id}_B \in {\rm Hom}(B, B) \subset {\rm Hom}(B, A_{X \times Y})$. Furthermore, recalling that $I_{X \times Y} \cong I_X \times I_Y$, we have
\begin{align}
&H^2(I_{X \times Y}, A_{X \times Y})\notag\\
&\cong H^2(I_X, A_{X \times Y}) \oplus H^2(I_Y, A_{X \times Y}) \oplus {\rm Hom}(H_1(I_X)\otimes H_1(I_Y), A_{X \times Y}). \label{ixy-eq}
\end{align}
Since the restriction homomorphism $H^2(\overline{\rm Inn}(X), A_{X \times Y}) \to H^2(I_{X \times Y}, A_{X \times Y})$ preserves the decompositions (\ref{h2-eq}) and (\ref{ixy-eq}), we may regard $f|_{I_{X \times Y}}$ as the sum of $f_X|_{I_X} \in H^2(I_X, A_X)$, $f_Y|_{I_Y} \in H^2(I_Y, A_Y)$, and the homomorphism $f_B: H_1(I_X) \otimes H_1(I_Y) \to B$ induced by the inclusions.
\par We claim that $f_B$ is zero. In fact, since $X$ is connected, the endomorphism $1-\sigma$ on $\overline{\rm Inn}_0(X)_{\rm ab}$ is surjective. Therefore, for any $g_X \in H_1(I_X)$ and $g_Y \in H_1(I_Y)$, we can take $h_X \in \overline{\rm Inn}_0(X)_{\rm ab}$ such that $(1 - \sigma)(h_X) = g_X$ to find that $(1-\sigma)(h_X \otimes g_Y) = g_X \otimes g_Y$ in $\overline{\rm Inn}_0(X)_{\rm ab} \otimes \overline{\rm Inn}_0(Y)_{\rm ab}$; remark that $\sigma(g_Y) = g_Y$ because $\bar{s}_{y_0}$ commutes with $I_Y$. Hence, we have $f_B = 0$ as claimed.
\par Thus, the extension of $I_{X \times Y}$ by $B$ is trivial and we have
$$\tilde{I}_{X \times Y} \cong \tilde{I}_X \times \tilde{I}_Y \times B,$$
where $\tilde{I}_X$ (resp. $\tilde{I}_Y$) is the extension of $I_X$ by $A_X$ (resp. $I_Y$ by $A_Y$) associated with $f_X|_{I_X}$ (resp. $f_Y|_{I_Y}$). Since $\tilde{I}_X$ (resp. $\tilde{I}_Y$) is isomorphic to ${\rm Stab}_{{\rm As}(X)}(x_0) \cap {\rm Ker}\,\epsilon$ (resp. ${\rm Stab}_{{\rm As}(Y)}(y_0) \cap {\rm Ker}\,\epsilon$), we take the abelianizations to conclude that $H_2^Q(X \times Y) \cong H_2^Q(X) \oplus H_2^Q(Y) \oplus B$.
\end{proof}
%-------------------------------------------------------------------------------------------------------------
\subsection*{Acknowledgements}
The author would like to express his gratitude to Takefumi Nosaka for the helpful comments on an earlier version of the paper. This work was supported by JSPS KAKENHI Grant Number JP20K14309.


\begin{thebibliography}{HD82}

%% Use the widest label as parameter above.
%% Reference items can be numbered or have labels of your choice, as below.
%% Arrange the items in the alphabetical order of names (and not in the order of labels).

%% In IMPAN journals, only the title is italicized; boldface is not used.
%% Do NOT give the issue number unless the issues are paginated separately, as in Uspekhi below.

%%%%%%%%%%% To ease editing, use normal size:

\normalsize
\baselineskip=17pt

%%%%%%%%%%%%%
%\bibitem{an-gr} Andruskiewitsch, N.; Gra$\tilde{\rm n}$a, M.; \textit{From racks to pointed Hopf algebras}. Adv. Math. \textbf{178} (2003), 177--243.
\bibitem{bimnp} Bakshi, R.P.; Ibarra, D.; Mukherjee, S.; Nosaka, T.; Przytycki, J.H.; \textit{Schur multipliers and second quandle homology}. J. Algebra 552 (2020), 52--67.
%\bibitem{brs} Brieskorn, E.; \textit{Automorphic sets and braids and singularities}. Contemp. Math., \textbf{78}, Amer. Math. Soc., 1988.
\bibitem{bro} Brown, K.S.; \textit{Cohomology of groups}. Graduate Texts in Mathematics, 87. Springer-Verlag, New York-Berlin, 1982.
\bibitem{cegs} Carter, J.S.; Elhamdadi, M.; Gra$\tilde{\rm n}$a, M.; Saito, M.; \textit{Cocycle knot invariants from quandle modules and generalized quandle homology}. Osaka J. Math. 42 (2005), 499--541.
\bibitem{cjkls} Carter, J.S.; Jelsovsky, D.; Kamada, S.; Langford, L.; Saito, M.; \textit{Quandle cohomology and state-sum invariants of knotted curves and surfaces}. Trans. Amer. Math. Soc. 355 (2003), 3947--3989.
\bibitem{cla} Clauwens, F.J.-B.J.; \textit{The adjoint group of an Alexander quandle}. arxiv:1011.1587
\bibitem{ed-li} Edmonds, A. L.; Livingston, C.; \textit{Symmetric representations of knot groups}.
Topology Appl. 18 (1984), 281--312.
%\bibitem{eis1} Eisermann, M.; \textit{Homological characterization of the unknot}. J. Pure Appl. Algebra \textbf{177} (2003), 131--157.
\bibitem{eis2} Eisermann, M.; \textit{Quandle coverings and their Galois correspondence}. Fund. Math. 225 (2014), 103--168.
\bibitem{el-ne} Elhamdadi, M.; Nelson, A.; \textit{Quandles---an introduction to the algebra of knots}. Student Mathematical Library, 74. Amer. Math. Soc., 2015.
\bibitem{hil} Hillman, J. A.; \textit{Sections of surface bundles}. Interactions between low-dimensional topology and mapping class groups, Geom. Topol. Monogr., 19 (2015), 1--19.
%\bibitem{ho-se} Hochschild, G.; Serre, J.-P.; \textit{Cohomology of group extensions}. Trans. Amer. Math. Soc. 74 (1953), 110--134.
\bibitem{joy} Joyce, D.; \textit{A classifying invariant of knots, the knot quandle}. J. Pure Appl. Algebra 23 (1982), 37--65.
\bibitem{li-ne} Litherland, R. A.; Nelson, S.; \textit{The Betti numbers of some finite racks}. J. Pure Appl. Algebra 178 (2003), 187--202.
\bibitem{mat} Matveev, S.V.; \textit{Distributive groupoids in knot theory}. (Russian) Mat. Sb. (N.S.) 119(161) (1982), 78--88. (English translation: Math. USSR-Sb. 47 (1984), 73--83.)
\bibitem{nos0} Nosaka, T.; \textit{On homotopy groups of quandle spaces and the quandle homotopy invariant of links. Topology Appl}. 158 (2011), 996--1011.
\bibitem{nos1} Nosaka, T.; \textit{Quandles and topological pairs. Symmetry, knots, and cohomology}. SpringerBriefs in Mathematics. Springer, Singapore, 2017.
\bibitem{nos2} Nosaka, T.; \textit{Central extensions of groups and adjoint groups of quandles}. Geometry and analysis of discrete groups and hyperbolic spaces, 167--184,
RIMS K$\hat{\rm o}$ky$\hat{\rm u}$roku Bessatsu, B66, Res. Inst. Math. Sci., Kyoto, 2017.
\end{thebibliography}
\end{document}